\documentclass[12pt,a4paper]{article}
\usepackage{filecontents}

\usepackage{pstricks} % for graph
\usepackage{pstricks-add} % for graph
\usepackage{booktabs} % for tables
\usepackage{multirow}
\usepackage{colortbl}
\usepackage{changepage}
\usepackage{amsmath}
\usepackage{graphicx}%
\usepackage{amsfonts}%
\usepackage{amssymb}
\newtheorem{theorem}{Theorem}

\newtheorem{lemma}[theorem]{Lemma}

\newenvironment{proof}[1][Proof]{\textbf{#1.} }{\ \rule{0.5em}{0.5em}}

\begin{document}

\title{A Sieve for Twin Primes}
\author{Jon S. Birdsey and Geza Schay}
\maketitle

\begin{abstract}
We present an algorithm analogous to the sieve of Eratosthenes that produces
the list of twin primes.

Next, we count the number of twin primes resulting from the construction with
two different heuristic arguments. The first method is essentially the same as
the one in \cite{1}. However, the second method is novel. It results in the
same asymptotic formula but it uses a simpler correction factor than \cite{1}.

Though we have no theory for the accuracy of our estimates, we compute them
both without and with the correction factor and they turn out to be close to
the actual counts up to $8009^{2}$.

\end{abstract}
\maketitle

A famous unproved problem in number theory is the twin prime conjecture, which
says that the number of twin \ primes, that is, primes that differ from each
other by 2, is infinite. A stronger form, still not completely proved, due to
Hardy and Wright (See \cite{1}, p. 372.) gives an asymptotic formula for the
number of twin primes under a real number $x$ as $x\rightarrow\infty.$

In this paper we present a new approach to the problem by proposing a new kind
of sieve, which produces the twin primes, and we use it to count their number
under $x$ with two different heuristic arguments. The first method is
essentially the same as the one in \cite{1}. However, the second method is
novel. It is based on Dirichlet's theorem on primes in arithmetic progressions
(see e.g. \cite{3}, p. 146), and results in the same asymptotic formula as the
one given in \cite{1}, but it uses a simpler correction factor than theirs.
Though we have no theory for the accuracy of our estimates, we compute them,
both without and with the correction factor, and they turn out to be close to
the actual counts up to $8009^{2}$.

\section{The Double Sieve.}

The following property of twin primes is well known:

\begin{lemma}
All twin primes greater than $3$ are of the form $6n+1$ and $6n-1$ for the
same $n\in\mathbb{N}.$
\end{lemma}

\begin{proof}
Clearly, the numbers greater than $3$ congruent to 0,2,3, and
$4\operatorname{mod}6$ are composite, and so all primes greater than $3$ must
be of the form $6n\pm1.$ Now for the same $n$ the two numbers $6n\pm1$ differ
by 2, but for $m\neq n,$ $6m\pm1$ and $6n\pm1$ differ by more than 2.
\end{proof}

Since we want to study twin primes and those are of the form $6n\pm1$, we sift
the set of such numbers. For the sake of convenience we call the two numbers
$6n\pm1$, for any positive integer $n,$ twins of each other, regardless of
whether they are primes or not.

For any prime $p\geq5,$ let $L_{p}$ denote the set of all positive integers of
the form $6n\pm1$ obtained by sifting out those with prime factors less than
$p$ together with their twins.

We construct these sets successively as follows.

First, $L_{5}=$ $\left\{  6n\pm1|n\in\mathbb{N}\right\}  ,$\ because this is
the set of all positive numbers with prime factors $\geq5.$ However, with an
eye on the upcoming construction of $L_{7},$ we write $L_{5}$ as a table
$T_{5}$ of two columns and five rows (see Table 1.), with the numbers $6n-1$
in the first column and the numbers $6n+1$ in the second column for
$n=1,\ldots,5,$ and consider the rest of $L_{5}$ as the numbers of this table
$+30k$ for $k=1,2,\ldots.$ Here $30=5\#.$\footnote{$n\#$ $=\prod_{p\leq n}p=n$-primorial}

To construct $L_{7}$, we proceed as follows: In $T_{5}$, we delete the
multiples of 5, that is, 5 and 25 (shown in dark orange) and the twins of
these multiples, that is, 7 and 23.(shown in light orange), and in $L_{5}$ we
delete all the numbers in the residue classes of these four numbers
$\operatorname{mod}30.$ Notice that, apart from the number 5, among the
deleted residue classes only those of 7 and 23 contain primes, and, except for
7, these are nontwin primes, because their twins are multiples of 5.

$T_{5}$ has 5 rows and 2 columns, and in each column we delete 2 numbers. So,
we have $2\left(  5-2\right)  =6$ undeleted numbers left. We write a table
$T_{7}$ for $L_{7}$ (see Table 2) consisting of six columns headed by the six
numbers that remained undeleted in $L_{5},$ and seven rows that are obtained
by continuing the residue classes $\operatorname{mod}30$ of the headers of the
columns. $L_{7}$ is the set of the numbers of this table $+210k$ for
$k=1,2,\ldots.$ Here $210=7\#.$

To construct $L_{11}$, we proceed similarly: In $L_{7}$, we delete the
multiples of 7 (shown in dark orange in $T_{7}$) and the twins of these
multiples (shown in light orange), plus all the numbers in the residue classes
of these twelve numbers $\operatorname{mod}7\#=\operatorname{mod}210.$ Thus
the deleted numbers include all the nontwin primes that are twins of multiples
of 7 and no other primes. $T_{7}$ has 7 rows and 6 columns, and, by the
Chinese remainder theorem, in each column we have a multiple of 7 and a twin
of a multiple of 7, which we delete. So, in $T_{7}$ we have $2\left(
5-2\right)  \left(  7-2\right)  =30$ undeleted numbers left.

From the undeleted numbers in $T_{7}$, we form the first row of $T_{11},$
shown in Table 3. Below this row we write 10 more rows by adding $210k$ for
$k=1,2,\ldots,10$ to each header. $T_{11}$ is the first complete set of
residues $\operatorname{mod}11\#=\operatorname{mod}2310$ of $L_{11}.$

We continue building $L_{p}$ for all successive primes in a similar fashion.

\begin{lemma}
For $p>3,$ every number $q<p^{2}-2$ in $L_{p}$ must be a twin prime, and every
twin prime $\geq p$ is in $L_{p}.$
\end{lemma}

\begin{proof}
Consider $L_{p}$ for some given $p.$ Then $p^{2}$ is the smallest composite
number\ with all factors $\geq p,$ and so no composite $q<p^{2}$ can be in $L_{p}.$

If $q<p^{2}-2$ is a nontwin prime $>3,$ then its twin is composite and
$<p^{2},$ hence $q$ cannot be in $L_{p}.$

So, if there are numbers in $L_{p}$ under $p^{2}-2,$ then they must be twin primes.

Furthermore, if $q$ is a twin prime $\geq p,$ then it is a member of $L_{p}.$
This is so, because in the construction of $L_{p}$ we have sifted out all
multiples of the primes $<p$ and their twins, but no twin prime $\geq p$ is
one of those.
\end{proof}

Thus, our construction provides a sieve for the twin primes analogous to the
sieve of Eratosthenes, as illustrated in Tables 1-3, which start with the twin
primes between $p$ and $p^{2}-2,$ for $p=5,7,11.$ We have no proof, however,
that the construction will give twin primes for every $p,$ no matter how
large. We have proved only that if there are numbers between $p$ and $p^{2}-2$
in $L_{p},$ then they are the twin primes of that interval.

Based on the construction above, in the next two sections we shall estimate
the number of twin primes in $L_{p}$ under $p^{2}-2$ in two ways. First, by
counting the number of deletions in each step of the construction and
computing the density of the undeleted numbers left that way.\ 

The second way of counting the number of twin primes in $L_{p}$ under
$p^{2}-2,$ is by counting the number of \emph{primes} deleted under $p^{2}$ in
the steps of the construction of $L_{p}$ and subtracting that from the total
number $\pi\left(  p^{2}\right)  $ of primes under $p^{2}.$

\section{Counting all deletions under $p^{2}.$}

In $T_{5}$\ we delete $2$ entries in each column, and so we keep $2\left(
5-2\right)  =6$ entries out of $5\#=30$ numbers. Hence the density of the
undeleted numbers in $T_{5}$ is $d_{5}=2\left(  5-2\right)  /5\#=1/5.$ These
numbers become the entries of $T_{7,}$ and so $d_{5}$ is also the density of
$T_{7}$ and of $L_{7}.$

In $T_{7}$ we delete $2$ entries out of 7 in each column, and so the density
of the undeleted numbers in $T_{7}$ becomes $d_{7}=\left(  7-2\right)
d_{5}/7=2\left(  5-2\right)  \left(  7-2\right)  /7\#=1/7,$ which is then also
the density of $T_{11}$ and of $L_{11}.$ Continuing in this way, and writing
$p_{n}$ for the $n$th prime, we obtain the density of $T_{p_{n}}$\ and
$L_{p_{n}},$ for $n=3,4,\ldots$ as%
\begin{equation}
d_{p_{n-1}}=\frac{2\prod_{k=2}^{n-1}\left(  p_{k}-2\right)  }{p_{n-1}\#}%
=\prod_{k=2}^{n-1}\frac{p_{k}-2}{p_{k}}.
\end{equation}

This density is essentially the same\footnote{\cite{1} gives the number of
pairs, we give the density of individual primes.} as the one on p. 372 in
\cite{1} and has the same shortcoming, namely that it is the density of
$L_{p_{n}}$ over the huge interval $\left[  1,p_{n}\#\right]  ,$ but we need
the density of $L_{p_{n}}$ over the much shorter interval $\left[  p_{n}%
,p_{n}^{2}\right]  ,$ where it is the same as the density of\ the twin primes
there. As detailed in \cite{1}, the corresponding ratio of the number of all
primes relative to $p_{n}\#$ in the short interval to that in the long
interval is known to be about $e^{\gamma}/2,$ by Mertens' theorem. (Thm. 429,
lc. Here $\gamma$ is Euler's constant.)\ Thus it is conjectured there that for
twin primes the corresponding correction factor should be $\left(  e^{\gamma
}/2\right)  ^{2}.$ This amounts to assuming the statistical independence of
the two numbers in a twin pair. (See \cite{2}.)

We try to avoid this problem by counting just the \emph{primes} deleted under
$p_{n}^{2}$ in the construction, since those are all nontwin primes. We do
this in the next section. Unfortunately, however, we have no theory to
determine the accuracy of our approximation. Empirically it turns out to be
very good, though, and with a plausible correction factor, which has no square
as the one in the first method, it becomes the same as the one above.

\section{Counting deleted primes.}

Let $N_{p}\left(  p_{k}\right)  =$ number of primes under $p^{2}$ deleted from
$L_{p_{k}}$ when we build $L_{p_{k+1}}$ from $L_{p_{k}}$ for $3<p_{k}<p.$

In $L_{5}$ there are 2 columns in the first period, and in each column there
is one entry that is the twin of a multiple of 5, namely 23 in the first
column and 7 in the second. Thus, the primes that we delete from $L_{5}$ are
the primes under $p^{2}$ of the two residue classes $7\operatorname{mod}30$
and $23\operatorname{mod}30$ plus the single prime $5.$ According to
Dirichlet's theorem on primes in arithmetic progressions (see e.g. \cite{3},
p. 146), each residue class that is relatively prime to the modulus has an
approximately equal share of the $\pi\left(  p^{2}\right)  $ primes under
$p^{2},$ that is, about $\pi\left(  p^{2}\right)  /\varphi\left(  30\right)  $
primes.\ Thus, for large $p$ we delete
\begin{equation}
N_{p}\left(  5\right)  \approx2\cdot\frac{\pi\left(  p^{2}\right)  }%
{\varphi\left(  30\right)  }+1\approx\frac{2}{\left(  3-1\right)  \left(
5-1\right)  }\pi\left(  p^{2}\right)  =\frac{1}{4}\pi\left(  p^{2}\right)
\end{equation}
primes from $L_{5}$ under $p^{2}$. This number is just an approximation,
because the fraction is correct only over integer multiples of the period and
Dirichlet's theorem is only asymptotically true. The same considerations apply
to the estimates below as well.

Similarly, from $L_{7}$ under $p^{2}$ we delete about
\begin{equation}
N_{p}\left(  7\right)  \approx2\left(  5-2\right)  \frac{\pi\left(
p^{2}\right)  }{\varphi\left(  210\right)  }=\frac{2\left(  5-2\right)
}{\left(  3-1\right)  \left(  5-1\right)  \left(  7-1\right)  }\pi\left(
p^{2}\right)  =\frac{1}{8}\pi\left(  p^{2}\right)
\end{equation}
primes, because $L_{7}$ has a period of 210, the first period has $2\left(
5-2\right)  =6$ columns, and in each column we delete one multiple of $7$ and
one twin of such a multiple, which is prime to 210. (All entries of $L_{7}$
are prime to $2,3,$ and $5,$ and the twin of a multiple of 7 in $L_{7}$\ is
prime also to 7, and so to $7\#.$)

From $L_{11}$ we delete about%
\begin{align}
N_{p}\left(  11\right)   &  \approx2\left(  5-2\right)  \left(  7-2\right)
\frac{\pi\left(  p^{2}\right)  }{\varphi\left(  2310\right)  }=\frac{2\left(
5-2\right)  \left(  7-2\right)  }{\left(  3-1\right)  \left(  5-1\right)
\left(  7-1\right)  \left(  11-1\right)  }\pi\left(  p^{2}\right) \nonumber\\
&  =\frac{\left(  5-2\right)  \left(  7-2\right)  }{\left(  5-1\right)
\left(  7-1\right)  \left(  11-1\right)  }\pi\left(  p^{2}\right)
=\frac{1}{16}\pi\left(  p^{2}\right)
\end{align}
primes under $p^{2}$.

Similarly, for general $p_{k}<p,$ with $k\geq4,$ the number of primes deleted
from $L_{p_{k}}$ under $p^{2}$ \ is
\begin{align}
N_{p}\left(  p_{k}\right)   &  \approx2\prod_{i=2}^{k-1}\left(  p_{i}%
-2\right)  \frac{\pi\left(  p^{2}\right)  }{\varphi\left(  p_{k}\#\right)
}=2\prod_{i=2}^{k-1}\frac{p_{i}-2}{p_{i}-1}\cdot\frac{1}{p_{k}-1}\pi\left(
p^{2}\right) \nonumber\\
&  =\prod_{i=3}^{k-1}\frac{p_{i}-2}{p_{i}-1}\cdot\frac{1}{p_{k}-1}\pi\left(
p^{2}\right)  =\prod_{i=3}^{k-1}\left(  1-\frac{1}{p_{i}-1}\right)
\cdot\frac{1}{p_{k}-1}\pi\left(  p^{2}\right) \nonumber\\
&  =\left[  \prod_{i=3}^{k-1}\left(  1-\frac{1}{p_{i}-1}\right)  -\prod
_{i=3}^{k-1}\left(  1-\frac{1}{p_{i}-1}\right)  \cdot\left(  1-\frac{1}%
{p_{k}-1}\right)  \right]  \pi\left(  p^{2}\right) \nonumber\\
&  =\left[  \prod_{i=3}^{k-1}\left(  1-\frac{1}{p_{i}-1}\right)  -\prod
_{i=3}^{k}\left(  1-\frac{1}{p_{i}-1}\right)  \right]  \pi\left(
p^{2}\right)  .
\end{align}

Thus, if we sum these expressions over $k,$ we get a telescoping sum and the
total number of primes deleted under $p_{n}^{2}$ is
\begin{equation}
\sum_{k=3}^{n}N_{p_{n}}\left(  p_{k}\right)  \approx\left[  1-\prod_{i=3}%
^{n}\left(  1-\frac{1}{p_{i}-1}\right)  \right]  \pi\left(  p_{n}^{2}\right)
,
\end{equation}
and so the number of twin primes left in $L_{p_{n}}$ under $p_{n}^{2},$ that
is, in the interval $\left[  p_{n},p_{n}^{2}\right]  ,$ is about
\begin{equation}
\prod_{i=3}^{n}\left(  1-\frac{1}{p_{i}-1}\right)  \pi\left(  p_{n}%
^{2}\right)  =\prod_{i=3}^{n}\left(  \frac{p_{i}-2}{p_{i}-1}\right)
\pi\left(  p_{n}^{2}\right)  . \label{count}%
\end{equation}

The second expression in Equation \ref{count} can also be interpreted as the
ratio of the number of twins relatively prime to $p_{n}\#$ to that of all
numbers relatively prime to $p_{n}\#$ in the interval $[5,p_{n}\#],$ times the
number of primes in $[1,p_{n}^{2}],$ (which has the same order of magnitude as
that in $\left[  p_{n},p_{n}^{2}\right]  ).$ Thus, while we did not assume
that the density of twin primes is the same in the short interval as in the
long one, we obtained a result that is equivalent to the ratio of the
densities of the twin relative primes and of all relative primes being the
same in the two intervals. (Clearly, in $\left[  p_{n},p_{n}^{2}\right]  $ the
numbers relatively prime to $p_{n}\#$ are the same as the primes there.)

Apparently, we can improve the estimate above by applying the same correction
factor $r$ that appears when comparing the numbers of \emph{all} relative
primes in the two intervals:
\begin{equation}
\pi\left(  p_{n}^{2}\right)  =rp_{n}^{2}\prod_{i=1}^{n}\frac{p_{i}-1}{p_{i}}.
\end{equation}
Here $p_{n}^{2}$ is the length of the short interval and the product is the
density of all relative primes in the long interval. Thus, $p_{n}^{2}%
\prod_{i=1}^{n}\frac{p_{i}-1}{p_{i}}$ would be the number of primes in the
interval$\ [1,p_{n}^{2}],$ if the density there were the same.

Hence%

\begin{equation}
r=\frac{\pi\left(  p_{n}^{2}\right)  }{p_{n}^{2}}\prod_{i=1}^{n}\frac{p_{i}%
}{p_{i}-1}.
\end{equation}
Denoting $p_{n}^{2}$ by $x,$ we can also write%
\begin{equation}
r=\frac{\pi\left(  x\right)  }{x}\prod_{p\leq\sqrt{x}}\frac{p}{p-1}.
\end{equation}

By Mertens' theorem,%
\begin{equation}
\prod_{p\leq\sqrt{x}}\frac{p-1}{p}\sim\frac{2e^{-\gamma}}{\log x},
\end{equation}
and by the prime number theorem,
\begin{equation}
\pi\left(  x\right)  \sim\frac{x}{\log x}.
\end{equation}
Thus,%
\begin{equation}
r\sim\frac{1}{\log x}\cdot\frac{\log x}{2e^{-\gamma}}=\frac{e^{\gamma}}{2}.
\end{equation}

Applying the correction factor $r$ to the estimate in Equation \ref{count}, we get%

\begin{align}
\pi_{2}\left(  x\right)   &  \approx r\prod_{i=3}^{n}\frac{p_{i}-2}{p_{i}%
-1}\pi\left(  p_{n}^{2}\right)  =\frac{\pi\left(  x\right)  }{x}\prod
_{p\leq\sqrt{x}}\frac{p}{p-1}\prod_{5\leq p\leq\sqrt{x}}\frac{p-2}{p-1}%
\pi\left(  x\right) \\
&  =\frac{\left(  \pi\left(  x\right)  \right)  ^{2}}{x}\cdot4\prod_{3\leq
p\leq\sqrt{x}}\frac{p\left(  p-2\right)  }{\left(  p-1\right)  ^{2}}.
\label{corr}%
\end{align}
In Table 4 we compare the estimate in Equation \ref{count} and this estimate
to the actual count of twin primes in the interval $\left[  p,p^{2}\right]  $.
As one can see, the numbers from Equation \ref{count} are pretty close to the
actual counts, but those from Equation \ref{corr} are much closer. (See also
Fig. 1.)

If we apply the asymptotic estimate $\pi\left(  x\right)  \sim x/\log x$ for
the number of all primes under $x,$ then from Equation \ref{corr} we obtain
\begin{equation}
\pi_{2}\left(  x\right)  \sim\frac{4x}{\log^{2}x}\prod_{3\leq p}\frac{p\left(
p-2\right)  }{\left(  p-1\right)  ^{2}},
\end{equation}
as an asymptotic estimate for the number of twin primes under $x.$ This result
equals the estimate in Equation 22.20.1 in \cite{1}, where the number of pairs
is counted, accounting for the factor of 2 in place of our 4 for the count of
the individual primes.

%
%
% her are tables 1 to 3 
%
% Table 1 -  generated by Excel2LaTeX
\begin{table}[htbp]
  \centering
    \begin{tabular}{rr}
        %\toprule
        \rowcolor[rgb]{ 1,  .6,  0} \multicolumn{1}{|r}{5} & \multicolumn{1}{r|}{\cellcolor[rgb]{ 1,  .8,  .6}7} \\
        \rowcolor[rgb]{ .812,  .906,  .961} \multicolumn{1}{|r}{11} & \multicolumn{1}{r|}{13} \\
        \rowcolor[rgb]{ .812,  .906,  .961} \multicolumn{1}{|r}{17} & \multicolumn{1}{r|}{19} \\
        \rowcolor[rgb]{ 1,  .8,  .6} \multicolumn{1}{|r}{23} & \multicolumn{1}{r|}{\cellcolor[rgb]{ 1,  .6,  0}25} \\
        \rowcolor[rgb]{ .812,  .906,  .961} \multicolumn{1}{|r}{29} & \multicolumn{1}{r|}{31} \\
        %\midrule
              &  \\
    \end{tabular}%
   \caption{ The Set T\textsubscript{5}.} 
  \label{tab:addlabel}%
\end{table}%
%
%
% Table 2 -  generated by Excel2LaTeX
\begin{table}[htbp]
  \centering
    \begin{tabular}{rrrrrr}
    %\toprule
    \rowcolor[rgb]{ .812,  .906,  .961} \multicolumn{1}{|r}{11} & 13    & \cellcolor[rgb]{ 1,  1,  1}17 & \cellcolor[rgb]{ 1,  1,  1}19 & 29    & \multicolumn{1}{r|}{31} \\
    \rowcolor[rgb]{ .812,  .906,  .961} \multicolumn{1}{|r}{41} & 43    & \cellcolor[rgb]{ 1,  .8,  .6}47 & \cellcolor[rgb]{ 1,  .6,  0}49 & 59    & \multicolumn{1}{r|}{61} \\
    \rowcolor[rgb]{ .812,  .906,  .961} \multicolumn{1}{|r}{71} & 73    & \cellcolor[rgb]{ 1,  .6,  0}77 & \cellcolor[rgb]{ 1,  .8,  .6}79 & \cellcolor[rgb]{ 1,  .8,  .6}89 & \multicolumn{1}{r|}{\cellcolor[rgb]{ 1,  .6,  0}91} \\
    \rowcolor[rgb]{ .812,  .906,  .961} \multicolumn{1}{|r}{101} & 103   & \cellcolor[rgb]{ 1,  1,  1}107 & \cellcolor[rgb]{ 1,  1,  1}109 & \cellcolor[rgb]{ 1,  .6,  0}119 & \multicolumn{1}{r|}{\cellcolor[rgb]{ 1,  .8,  .6}121} \\
    \rowcolor[rgb]{ 1,  .8,  .6} \multicolumn{1}{|r}{131} & \cellcolor[rgb]{ 1,  .6,  0}133 & \cellcolor[rgb]{ 1,  1,  1}137 & \cellcolor[rgb]{ 1,  1,  1}139 & \cellcolor[rgb]{ .812,  .906,  .961}149 & \multicolumn{1}{r|}{\cellcolor[rgb]{ .812,  .906,  .961}151} \\
    \rowcolor[rgb]{ 1,  .6,  0} \multicolumn{1}{|r}{161} & \cellcolor[rgb]{ 1,  .8,  .6}163 & \cellcolor[rgb]{ 1,  1,  1}167 & \cellcolor[rgb]{ 1,  1,  1}169 & \cellcolor[rgb]{ .812,  .906,  .961}179 & \multicolumn{1}{r|}{\cellcolor[rgb]{ .812,  .906,  .961}181} \\
    \rowcolor[rgb]{ .812,  .906,  .961} \multicolumn{1}{|r}{191} & 193   & \cellcolor[rgb]{ 1,  1,  1}197 & \cellcolor[rgb]{ 1,  1,  1}199 & 209   & \multicolumn{1}{r|}{211} \\
    %\midrule
          &       &       &       &       &  \\
    \end{tabular}%
   \caption{The Set T\textsubscript{7}.} 
  \label{tab:addlabel}%
\end{table}%
%
%
% Table 3 -  generated by Excel2LaTeX
\begin{adjustwidth}{-16em}{-16em}
 \setlength{\textwidth}{490pt}%
\begin{table}[htbp]
  \centering
  
  \resizebox{.95\textwidth}{!}{%
    \begin{tabular}{rrrrrrrrrrrrrrrrrrrrrrrrrrrrrr}
    %\toprule
    \rowcolor[rgb]{ 1,  .6,  0} \multicolumn{1}{|r}{11} & \cellcolor[rgb]{ 1,  .8,  .6}13 & \cellcolor[rgb]{ 1,  1,  1}17 & \cellcolor[rgb]{ 1,  1,  1}19 & \cellcolor[rgb]{ .812,  .906,  .961}29 & \cellcolor[rgb]{ .812,  .906,  .961}31 & \cellcolor[rgb]{ 1,  1,  1}41 & \cellcolor[rgb]{ 1,  1,  1}43 & \cellcolor[rgb]{ .812,  .906,  .961}59 & \cellcolor[rgb]{ .812,  .906,  .961}61 & \cellcolor[rgb]{ 1,  1,  1}71 & \cellcolor[rgb]{ 1,  1,  1}73 & \cellcolor[rgb]{ .812,  .906,  .961}101 & \cellcolor[rgb]{ .812,  .906,  .961}103 & \cellcolor[rgb]{ 1,  1,  1}107 & \cellcolor[rgb]{ 1,  1,  1}109 & \cellcolor[rgb]{ .812,  .906,  .961}137 & \cellcolor[rgb]{ .812,  .906,  .961}139 & \cellcolor[rgb]{ 1,  1,  1}149 & \cellcolor[rgb]{ 1,  1,  1}151 & \cellcolor[rgb]{ .812,  .906,  .961}167 & \cellcolor[rgb]{ .812,  .906,  .961}169 & \cellcolor[rgb]{ 1,  1,  1}179 & \cellcolor[rgb]{ 1,  1,  1}181 & \cellcolor[rgb]{ .812,  .906,  .961}191 & \cellcolor[rgb]{ .812,  .906,  .961}193 & \cellcolor[rgb]{ 1,  1,  1}197 & \cellcolor[rgb]{ 1,  1,  1}199 & 209   & \multicolumn{1}{r|}{\cellcolor[rgb]{ 1,  .8,  .6}211} \\
    \rowcolor[rgb]{ .812,  .906,  .961} \multicolumn{1}{|r}{221} & 223   & \cellcolor[rgb]{ 1,  1,  1}227 & \cellcolor[rgb]{ 1,  1,  1}229 & 239   & 241   & \cellcolor[rgb]{ 1,  .8,  .6}251 & \cellcolor[rgb]{ 1,  .6,  0}253 & 269   & 271   & \cellcolor[rgb]{ 1,  1,  1}281 & \cellcolor[rgb]{ 1,  1,  1}283 & 311   & 313   & \cellcolor[rgb]{ 1,  .8,  .6}317 & \cellcolor[rgb]{ 1,  .6,  0}319 & 347   & 349   & \cellcolor[rgb]{ 1,  1,  1}359 & \cellcolor[rgb]{ 1,  1,  1}361 & 377   & 379   & \cellcolor[rgb]{ 1,  1,  1}389 & \cellcolor[rgb]{ 1,  1,  1}391 & 401   & 403   & \cellcolor[rgb]{ 1,  .6,  0}407 & \cellcolor[rgb]{ 1,  .8,  .6}409 & 419   & \multicolumn{1}{r|}{421} \\
    \rowcolor[rgb]{ .812,  .906,  .961} \multicolumn{1}{|r}{431} & 433   & \cellcolor[rgb]{ 1,  1,  1}437 & \cellcolor[rgb]{ 1,  1,  1}439 & \cellcolor[rgb]{ 1,  .8,  .6}449 & \cellcolor[rgb]{ 1,  .6,  0}451 & \cellcolor[rgb]{ 1,  1,  1}461 & \cellcolor[rgb]{ 1,  1,  1}463 & 479   & 481   & \cellcolor[rgb]{ 1,  1,  1}491 & \cellcolor[rgb]{ 1,  1,  1}493 & 521   & 523   & \cellcolor[rgb]{ 1,  1,  1}527 & \cellcolor[rgb]{ 1,  1,  1}529 & 557   & 559   & \cellcolor[rgb]{ 1,  1,  1}569 & \cellcolor[rgb]{ 1,  1,  1}571 & 587   & 589   & \cellcolor[rgb]{ 1,  1,  1}599 & \cellcolor[rgb]{ 1,  1,  1}601 & 611   & 613   & \cellcolor[rgb]{ 1,  1,  1}617 & \cellcolor[rgb]{ 1,  1,  1}619 & 629   & \multicolumn{1}{r|}{631} \\
    \rowcolor[rgb]{ .812,  .906,  .961} \multicolumn{1}{|r}{641} & 643   & \cellcolor[rgb]{ 1,  .8,  .6}647 & \cellcolor[rgb]{ 1,  .6,  0}649 & 659   & 661   & \cellcolor[rgb]{ 1,  .6,  0}671 & \cellcolor[rgb]{ 1,  .8,  .6}673 & 689   & 691   & \cellcolor[rgb]{ 1,  1,  1}701 & \cellcolor[rgb]{ 1,  1,  1}703 & 731   & 733   & \cellcolor[rgb]{ 1,  .6,  0}737 & \cellcolor[rgb]{ 1,  .8,  .6}739 & 767   & 769   & \cellcolor[rgb]{ 1,  .8,  .6}779 & \cellcolor[rgb]{ 1,  .6,  0}781 & 797   & 799   & \cellcolor[rgb]{ 1,  1,  1}809 & \cellcolor[rgb]{ 1,  1,  1}811 & 821   & 823   & \cellcolor[rgb]{ 1,  1,  1}827 & \cellcolor[rgb]{ 1,  1,  1}829 & 839   & \multicolumn{1}{r|}{841} \\
    \rowcolor[rgb]{ .812,  .906,  .961} \multicolumn{1}{|r}{851} & 853   & \cellcolor[rgb]{ 1,  1,  1}857 & \cellcolor[rgb]{ 1,  1,  1}859 & \cellcolor[rgb]{ 1,  .6,  0}869 & \cellcolor[rgb]{ 1,  .8,  .6}871 & \cellcolor[rgb]{ 1,  1,  1}881 & \cellcolor[rgb]{ 1,  1,  1}883 & 899   & 901   & \cellcolor[rgb]{ 1,  .8,  .6}911 & \cellcolor[rgb]{ 1,  .6,  0}913 & 941   & 943   & \cellcolor[rgb]{ 1,  1,  1}947 & \cellcolor[rgb]{ 1,  1,  1}949 & \cellcolor[rgb]{ 1,  .8,  .6}977 & \cellcolor[rgb]{ 1,  .6,  0}979 & \cellcolor[rgb]{ 1,  1,  1}989 & \cellcolor[rgb]{ 1,  1,  1}991 & 1007  & 1009  & \cellcolor[rgb]{ 1,  1,  1}1019 & \cellcolor[rgb]{ 1,  1,  1}1021 & 1031  & 1033  & \cellcolor[rgb]{ 1,  1,  1}1037 & \cellcolor[rgb]{ 1,  1,  1}1039 & 1049  & \multicolumn{1}{r|}{1051} \\
    \rowcolor[rgb]{ .812,  .906,  .961} \multicolumn{1}{|r}{1061} & 1063  & \cellcolor[rgb]{ 1,  .6,  0}1067 & \cellcolor[rgb]{ 1,  .8,  .6}1069 & 1079  & 1081  & \cellcolor[rgb]{ 1,  1,  1}1091 & \cellcolor[rgb]{ 1,  1,  1}1093 & \cellcolor[rgb]{ 1,  .8,  .6}1109 & \cellcolor[rgb]{ 1,  .6,  0}1111 & \cellcolor[rgb]{ 1,  1,  1}1121 & \cellcolor[rgb]{ 1,  1,  1}1123 & 1151  & 1153  & \cellcolor[rgb]{ 1,  1,  1}1157 & \cellcolor[rgb]{ 1,  1,  1}1159 & 1187  & 1189  & \cellcolor[rgb]{ 1,  .6,  0}1199 & \cellcolor[rgb]{ 1,  .8,  .6}1201 & 1217  & 1219  & \cellcolor[rgb]{ 1,  1,  1}1229 & \cellcolor[rgb]{ 1,  1,  1}1231 & \cellcolor[rgb]{ 1,  .8,  .6}1241 & \cellcolor[rgb]{ 1,  .6,  0}1243 & \cellcolor[rgb]{ 1,  1,  1}1247 & \cellcolor[rgb]{ 1,  1,  1}1249 & 1259  & \multicolumn{1}{r|}{1261} \\
    \rowcolor[rgb]{ .812,  .906,  .961} \multicolumn{1}{|r}{1271} & 1273  & \cellcolor[rgb]{ 1,  1,  1}1277 & \cellcolor[rgb]{ 1,  1,  1}1279 & 1289  & 1291  & \cellcolor[rgb]{ 1,  1,  1}1301 & \cellcolor[rgb]{ 1,  1,  1}1303 & 1319  & 1321  & \cellcolor[rgb]{ 1,  .6,  0}1331 & \cellcolor[rgb]{ 1,  .8,  .6}1333 & 1361  & 1363  & \cellcolor[rgb]{ 1,  1,  1}1367 & \cellcolor[rgb]{ 1,  1,  1}1369 & \cellcolor[rgb]{ 1,  .6,  0}1397 & \cellcolor[rgb]{ 1,  .8,  .6}1399 & \cellcolor[rgb]{ 1,  1,  1}1409 & \cellcolor[rgb]{ 1,  1,  1}1411 & 1427  & 1429  & \cellcolor[rgb]{ 1,  .6,  0}1439 & \cellcolor[rgb]{ 1,  .6,  0}1441 & 1451  & 1453  & \cellcolor[rgb]{ 1,  1,  1}1457 & \cellcolor[rgb]{ 1,  1,  1}1459 & 1469  & \multicolumn{1}{r|}{1471} \\
    \rowcolor[rgb]{ .812,  .906,  .961} \multicolumn{1}{|r}{1481} & 1483  & \cellcolor[rgb]{ 1,  1,  1}1487 & \cellcolor[rgb]{ 1,  1,  1}1489 & 1499  & 1501  & \cellcolor[rgb]{ 1,  1,  1}1511 & \cellcolor[rgb]{ 1,  1,  1}1513 & \cellcolor[rgb]{ 1,  .6,  0}1529 & \cellcolor[rgb]{ 1,  .8,  .6}1531 & \cellcolor[rgb]{ 1,  1,  1}1541 & \cellcolor[rgb]{ 1,  1,  1}1543 & \cellcolor[rgb]{ 1,  .8,  .6}1571 & \cellcolor[rgb]{ 1,  .6,  0}1573 & \cellcolor[rgb]{ 1,  1,  1}1577 & \cellcolor[rgb]{ 1,  1,  1}1579 & 1607  & 1609  & \cellcolor[rgb]{ 1,  1,  1}1619 & \cellcolor[rgb]{ 1,  1,  1}1621 & \cellcolor[rgb]{ 1,  .8,  .6}1637 & \cellcolor[rgb]{ 1,  .6,  0}1639 & \cellcolor[rgb]{ 1,  1,  1}1649 & \cellcolor[rgb]{ 1,  1,  1}1651 & \cellcolor[rgb]{ 1,  .6,  0}1661 & \cellcolor[rgb]{ 1,  .8,  .6}1663 & \cellcolor[rgb]{ 1,  1,  1}1667 & \cellcolor[rgb]{ 1,  1,  1}1669 & 1679  & \multicolumn{1}{r|}{1681} \\
    \rowcolor[rgb]{ .812,  .906,  .961} \multicolumn{1}{|r}{1691} & 1693  & \cellcolor[rgb]{ 1,  1,  1}1697 & \cellcolor[rgb]{ 1,  1,  1}1699 & 1709  & 1711  & \cellcolor[rgb]{ 1,  1,  1}1721 & \cellcolor[rgb]{ 1,  1,  1}1723 & 1739  & 1741  & \cellcolor[rgb]{ 1,  1,  1}1751 & \cellcolor[rgb]{ 1,  1,  1}1753 & 1781  & 1783  & \cellcolor[rgb]{ 1,  1,  1}1787 & \cellcolor[rgb]{ 1,  1,  1}1789 & 1817  & 1819  & \cellcolor[rgb]{ 1,  1,  1}1829 & \cellcolor[rgb]{ 1,  1,  1}1831 & 1847  & 1849  & \cellcolor[rgb]{ 1,  .6,  0}1859 & \cellcolor[rgb]{ 1,  .8,  .6}1861 & 1871  & 1873  & \cellcolor[rgb]{ 1,  1,  1}1877 & \cellcolor[rgb]{ 1,  1,  1}1879 & 1889  & \multicolumn{1}{r|}{1891} \\
    \rowcolor[rgb]{ 1,  .8,  .6} \multicolumn{1}{|r}{1901} & \cellcolor[rgb]{ 1,  .6,  0}1903 & \cellcolor[rgb]{ 1,  1,  1}1907 & \cellcolor[rgb]{ 1,  1,  1}1909 & \cellcolor[rgb]{ .812,  .906,  .961}1919 & \cellcolor[rgb]{ .812,  .906,  .961}1921 & \cellcolor[rgb]{ 1,  1,  1}1931 & \cellcolor[rgb]{ 1,  1,  1}1933 & \cellcolor[rgb]{ .812,  .906,  .961}1949 & \cellcolor[rgb]{ .812,  .906,  .961}1951 & \cellcolor[rgb]{ 1,  1,  1}1961 & \cellcolor[rgb]{ 1,  1,  1}1963 & \cellcolor[rgb]{ 1,  .6,  0}1991 & 1993  & \cellcolor[rgb]{ 1,  1,  1}1997 & \cellcolor[rgb]{ 1,  1,  1}1999 & \cellcolor[rgb]{ .812,  .906,  .961}2027 & \cellcolor[rgb]{ .812,  .906,  .961}2029 & \cellcolor[rgb]{ 1,  1,  1}2039 & \cellcolor[rgb]{ 1,  1,  1}2041 & \cellcolor[rgb]{ 1,  .6,  0}2057 & 2059  & \cellcolor[rgb]{ 1,  1,  1}2069 & \cellcolor[rgb]{ 1,  1,  1}2071 & \cellcolor[rgb]{ .812,  .906,  .961}2081 & \cellcolor[rgb]{ .812,  .906,  .961}2083 & \cellcolor[rgb]{ 1,  1,  1}2087 & \cellcolor[rgb]{ 1,  1,  1}2089 & 2099  & \multicolumn{1}{r|}{\cellcolor[rgb]{ 1,  .6,  0}2101} \\
    \rowcolor[rgb]{ .812,  .906,  .961} \multicolumn{1}{|r}{2111} & 2113  & \cellcolor[rgb]{ 1,  1,  1}2117 & \cellcolor[rgb]{ 1,  1,  1}2119 & 2129  & 2131  & \cellcolor[rgb]{ 1,  1,  1}2141 & \cellcolor[rgb]{ 1,  1,  1}2143 & 2159  & 2161  & \cellcolor[rgb]{ 1,  1,  1}2171 & \cellcolor[rgb]{ 1,  1,  1}2173 & 2201  & 2203  & \cellcolor[rgb]{ 1,  1,  1}2207 & \cellcolor[rgb]{ 1,  1,  1}2209 & 2237  & 2239  & \cellcolor[rgb]{ 1,  1,  1}2249 & \cellcolor[rgb]{ 1,  1,  1}2251 & 2267  & 2269  & \cellcolor[rgb]{ 1,  1,  1}2279 & \cellcolor[rgb]{ 1,  1,  1}2281 & 2291  & 2293  & \cellcolor[rgb]{ 1,  .8,  .6}2297 & \cellcolor[rgb]{ 1,  .6,  0}2299 & 2309  & \multicolumn{1}{r|}{2311} \\
    %\midrule
          &       &       &       &       &       &       &       &       &       &       &       &       &       &       &       &       &       &       &       &       &       &       &       &       &       &       &       &       &  \\
    \end{tabular}%
    }
    \caption{The Set T\textsubscript{11}.} 
  \label{tab:addlabel}%
\end{table}%
 \setlength{\textwidth}{390pt}%
 \end{adjustwidth}
%
% here is table 4
%
\begin{table}[htbp]
  \centering
    \begin{tabular}{|c|c|c|c|c|}
    %\toprule
    \hline
    \multicolumn{1}{|c|}{\textit{p}} & \multicolumn{1}{c|}{Actual} & \multicolumn{1}{c|}{Equation 7} & \multicolumn{1}{c|}{\textit{r}} & \multicolumn{1}{c|}{Equation 15} \\
    %\midrule
    \hline
    \multicolumn{1}{|c|}{101} & \multicolumn{1}{c|}{404} & \multicolumn{1}{c|}{394} & \multicolumn{1}{c|}{1.03975} & \multicolumn{1}{c|}{410} \\
 \hline
    \multicolumn{1}{|c|}{199} & \multicolumn{1}{c|}{1150} & \multicolumn{1}{c|}{1143} & \multicolumn{1}{c|}{1.01694} & \multicolumn{1}{c|}{1162} \\
 \hline
    \multicolumn{1}{|c|}{307} & \multicolumn{1}{c|}{2288} & \multicolumn{1}{c|}{2332} & \multicolumn{1}{c|}{0.99588} & \multicolumn{1}{c|}{2323} \\
 \hline
    \multicolumn{1}{|c|}{401} & \multicolumn{1}{c|}{3578} & \multicolumn{1}{c|}{3618} & \multicolumn{1}{c|}{0.99050} & \multicolumn{1}{c|}{3683} \\
 \hline
    \multicolumn{1}{|c|}{503} & \multicolumn{1}{c|}{5170} & \multicolumn{1}{c|}{5263} & \multicolumn{1}{c|}{0.98667} & \multicolumn{1}{c|}{5193} \\
 \hline
    \multicolumn{1}{|c|}{601} & \multicolumn{1}{c|}{6974} & \multicolumn{1}{c|}{7103} & \multicolumn{1}{c|}{0.98036} & \multicolumn{1}{c|}{6964} \\
 \hline
    \multicolumn{1}{|c|}{701} & \multicolumn{1}{c|}{8946} & \multicolumn{1}{c|}{9186} & \multicolumn{1}{c|}{0.97882} & \multicolumn{1}{c|}{8992} \\
 \hline
    \multicolumn{1}{|c|}{797} & \multicolumn{1}{c|}{11128} & \multicolumn{1}{c|}{11426} & \multicolumn{1}{c|}{0.97493} & \multicolumn{1}{c|}{11140} \\
 \hline
    \multicolumn{1}{|c|}{907} & \multicolumn{1}{c|}{13674} & \multicolumn{1}{c|}{14223} & \multicolumn{1}{c|}{0.97287} & \multicolumn{1}{c|}{13837} \\
 \hline
    \multicolumn{1}{|c|}{1009} & \multicolumn{1}{c|}{16556} & \multicolumn{1}{c|}{17053} & \multicolumn{1}{c|}{0.97038} & \multicolumn{1}{c|}{16548} \\
 \hline
    \multicolumn{1}{|c|}{1999} & \multicolumn{1}{c|}{53556} & \multicolumn{1}{c|}{55038} & \multicolumn{1}{c|}{0.96144} & \multicolumn{1}{c|}{52916} \\
 \hline
    \multicolumn{1}{|c|}{3001} & \multicolumn{1}{c|}{107610} & \multicolumn{1}{c|}{111342} & \multicolumn{1}{c|}{0.95734} & \multicolumn{1}{c|}{106592} \\
 \hline
    \multicolumn{1}{|c|}{4001} & \multicolumn{1}{c|}{176914} & \multicolumn{1}{c|}{184081} & \multicolumn{1}{c|}{0.95390} & \multicolumn{1}{c|}{175595} \\
 \hline
    \multicolumn{1}{|c|}{5003} & \multicolumn{1}{c|}{261086} & \multicolumn{1}{c|}{272412} & \multicolumn{1}{c|}{0.95202} & \multicolumn{1}{c|}{259343} \\
 \hline
    \multicolumn{1}{|c|}{6007} & \multicolumn{1}{c|}{358978} & \multicolumn{1}{c|}{375972} & \multicolumn{1}{c|}{0.95005} & \multicolumn{1}{c|}{357192} \\
 \hline
    \multicolumn{1}{|c|}{7001} & \multicolumn{1}{c|}{469528} & \multicolumn{1}{c|}{492326} & \multicolumn{1}{c|}{0.94945} & \multicolumn{1}{c|}{467437} \\
 \hline
    \multicolumn{1}{|c|}{8009} & \multicolumn{1}{c|}{594636} & \multicolumn{1}{c|}{625062} & \multicolumn{1}{c|}{0.94773} & \multicolumn{1}{c|}{592388} \\
    %\midrule
    \hline
    \end{tabular}%
    \caption{Count of Twin Primes between \textit{p }and \textit{p\textsuperscript{2}}}
  \label{tab:addLabel}%
\end{table}%

%
% Here is the code for the difference chart
% Date for chart is at the top of this file
%
\psset{xunit=.0015cm, yunit=1.03cm, ylabelFactor=\% }
\begin{pspicture}[](-1000,-3.0)(9000,8.0)
	\psaxes[Dx=1000,Dy=1.0,tickstyle=bottom](0,0)(0,-1.5)(8500,6.0)
	\psline[linewidth=.05,linecolor=blue](5000,2.0)(6000,2.0)
	\rput(-1000,2.5){\rotateleft{\scriptsize\rmfamily \% Differnce}}
	\uput[0](4000,-1.1){\scriptsize\rmfamily Prime }
	\rput[l](6100,2.0){\small\rmfamily Equation 7}
	\psline[linewidth=.05,linecolor=red](5000,1.0)(6000,1.0)
	\rput[l](6100,1.0){\small\rmfamily Equation 15}
	\rput[l](1500,- 2.0){\small\rmfamily Figure 1.  \% Difference from Actual}
	\fileplot[linewidth=1.5pt,linecolor=blue]{eq7.dat}
	\fileplot[linewidth=1.5pt,linecolor=red]{eq15.dat}
\end{pspicture}
\end{document}